\documentclass[12pt]{article}
\usepackage{graphicx}
\usepackage{amsmath,amsthm,amssymb,amsfonts,euscript,enumerate}
\usepackage{amsthm}
\usepackage{color,wrapfig}
\usepackage{pxfonts}
\usepackage{verbatim}
\newtheorem{thm}{Theorem}[section]
\newtheorem{cor}[thm]{Corollary}
\newtheorem{lem}[thm]{Lemma}

\newcommand{\3}{\varepsilon}
\newcommand{\4}{\widetilde}

\def\ni{\noindent}
\begin{document}
\title{Non-existence of radially symmetric singular self-similar solutions of the fast diffusion equation}
\author{Shu-Yu Hsu\\
%\thanks{ }\\
Department of Mathematics\\
National Chung Cheng University\\
168 University Road, Min-Hsiung\\
Chia-Yi 621, Taiwan, R.O.C.\\
e-mail: shuyu.sy@gmail.com}
\date{April 6, 2025}
\smallbreak \maketitle
\begin{abstract}
Let $n\ge 3$, $0<m<\frac{n-2}{n}$, $\gamma>0$ and $\eta>0$. Suppose either (i) $\alpha\ne 0$ and $\beta=0$ or (ii) $\alpha\in\mathbb{R}$ and $\beta\ne 0$ holds. We will study the elliptic equation $\Delta (f^m/m)+\alpha f+\beta x\cdot\nabla f=0$, $f>0$, in $\mathbb{R}^n\setminus\{0\}$ with $\underset{\substack{r\to 0}}{\lim}\,r^{\gamma}f(r)=\eta$. This equation arises from the study of the singular self-similar solutions of the fast diffusion equation which blow up at the origin. We will prove that if there exists a radially symmetric singular solution of the above elliptic equation, then either $\gamma=\frac{2}{1-m}$ and $\alpha>\frac{2\beta}{1-m}$ or $\gamma>\frac{2}{1-m}$, $\beta\ne 0$ and $\gamma=\alpha/\beta$. As a consequence we obtain the non-existence of radially symmetric self-similar solution of the fast diffusion equation $u_t=\Delta (u^m/m)$, $u>0$, 
which blows up at the origin with rate $|x|^{-\gamma}$ when either $0<\gamma\ne\frac{2}{1-m}$ and $\gamma\ne\alpha/\beta$, $\alpha\in\mathbb{R}$ and $\beta\ne 0$ or $\gamma=\frac{2}{1-m}$ and $\left(\alpha-\frac{2\beta}{1-m}\right)\eta^{1-m}\ne\frac{2(n-2-nm)}{(1-m)^2}$ holds.
\end{abstract}

\vskip 0.2truein



\vskip 0.2truein
\setcounter{section}{0}

\section{Introduction}
\setcounter{equation}{0}
\setcounter{thm}{0}

Let $n\ge 3$ and $0<m<\frac{n-2}{n}$. We will study the non-existence of  radially symmetric singular solutions of the fast diffusion equation, 
\begin{equation}\label{fde}
u_t=\Delta (u^m/m),\quad u>0,
\end{equation}
which blow up at the origin. The equation \eqref{fde} arises in many physical models and in the study of Yamabe flow. When $m>1$, \eqref{fde} is called the porous medium equation which arises in the modelling of gases passing through porous media and oil passing through sand  \cite{Ar}.  When $m=1$, it is the heat equation. When  $n\ge 3$ and $m=\frac{n-2}{n+2}$, \eqref{fde} arises in the study of the Yamabe flow \cite{DS1}, \cite{DS2}, \cite{PS}. 
   
Recently there is a lot of study of \eqref{fde} for the case $0<m<1$ by P.~Daskalopoulos, S.Y.~Hsu, K.M.~Hui, T.~Jin, Sunghoon Kim, Jinwan Park, M. del Pino, M.~S\'aez,  N.~Sesum, J.~Takahashi, J.L. Vazquez, H.~Yamamoto, E.~Yanagida, M.~Winkler and J.~Xiong, etc. \cite{DS1}, \cite{DS2}, \cite{H1}, \cite{H2}, \cite{HK1}, \cite{HK2}, \cite{Hs}, \cite{HP}, \cite{JX}, \cite{PS}, \cite{TY}, \cite{VW1}, \cite{VW2}. We refer the reader to the book \cite{V} by  J.L.~Vazquez for some recent results for the equation \eqref{fde}.

The study of self-similar solutions of \eqref{fde} is important because the asymptotic behaviour of the solutions of \eqref{fde} with appropriate initial values usually approach the self-similar solutions of \eqref{fde} after some rescaling. Let $n\ge 3$, $0<m<\frac{n-2}{n}$ and $\alpha,\beta\in\mathbb{R}$. Then the function
\begin{equation*}\label{forward-self-soln}
U(x,t)=t^{-\alpha}f(t^{-\beta}x)
\end{equation*}
is a forward self-similar solution of \eqref{fde} in $(\mathbb{R}^n\setminus\{0\})\times (0,\infty)$ if and only if $f$ satisfies
\begin{equation}\label{elliptic-eqn}
\Delta (f^m/m)+\alpha f+\beta x\cdot\nabla f=0,\quad f>0,\quad\mbox{ in }\mathbb{R}^n\setminus\{0\}
\end{equation}
with
\begin{equation}\label{alpha-beta-forward-sss}
\alpha=\frac{2\beta-1}{1-m}.
\end{equation}
Similarly the function
\begin{equation*}\label{backward-self-soln}
V(x,t)=(T-t)^{\alpha}f((T-t)^{\beta}x)
\end{equation*}
is a backward self-similar solution of \eqref{fde} in $(\mathbb{R}^n\setminus\{0\})\times (-\infty,T)$ if and only if $f$ satisfies \eqref{elliptic-eqn} with
\begin{equation}\label{alpha-beta-backward-sss}
\alpha=\frac{2\beta+1}{1-m}
\end{equation}
and the function
\begin{equation*}\label{external-self-soln}
W(x,t)=e^{\alpha t}f(e^{\beta x})
\end{equation*}
is an eternal self-similar solution of \eqref{fde} in $(\mathbb{R}^n\setminus\{0\})\times (-\infty,\infty)$ if and only if $f$ satisfies \eqref{elliptic-eqn} with
\begin{equation*}
\alpha=\frac{2\beta}{1-m}.
\end{equation*}
Hence the study of the existence of singular self-similar solutions of \eqref{fde} which blow up at the origin is equivalent to the study of the existence of singular solutions of \eqref{elliptic-eqn} which blow up at the origin. 

We will now assume that $n\ge 3$ and $0<m<\frac{n-2}{n}$ for the rest of the paper.
In the paper \cite{HKs} K.M.~Hui and Soojung~Kim proved that for any $\eta>0$ and $\alpha<0$, $\beta<0$, $\rho_1>0$, satisfying 
\begin{equation*}
\alpha=\frac{2\beta-\rho_1}{1-m}
\end{equation*}
 and 
\begin{equation*}
\frac{2}{1-m}<\frac{\alpha}{\beta}<\frac{n-2}{m}
\end{equation*}
there exists a unique radially symmetric singular solution of \eqref{elliptic-eqn} that satisfies 
\begin{equation}\label{blow-up-rate-at-origin}
\lim_{|x|\to 0}|x|^{\alpha/\beta}f(x)=\eta.
\end{equation}
On the other hand in \cite{H2} K.M.~Hui proved that for any $\eta>0$, $\rho_1>0$, $\beta\ge\frac{m\rho_1}{n-2-nm}$ and $\alpha$ satisfying 
\begin{equation*}
\alpha=\frac{2\beta+\rho_1}{1-m}
\end{equation*} 
there exists a radially symmetric singular solution of \eqref{elliptic-eqn} that satisfies 
\eqref{blow-up-rate-at-origin} and such solution is unique for $\beta\ge\frac{\rho_1}{n-2-nm}$. 
Note that the function \cite{DS1},
\begin{equation*}
B(x,t)=\left(\frac{C_{\ast}(T-t)}{|x|^2}\right)^{\frac{1}{1-m}}
\end{equation*}
where 
\begin{equation*}
C_{\ast}=\frac{2(n-2-nm)}{1-m}
\end{equation*}
is a singular radially symmetric self-similar solution of \eqref{fde}. Since
\begin{equation*}
B(x,t)=(T-t)^{\alpha}\left(\frac{C_{\ast}}{((T-t)^{\beta}|x|)^2}\right)^{\frac{1}{1-m}}
\end{equation*}
for any constants $\beta>0$ and $\alpha>0$ satisfying \eqref{alpha-beta-backward-sss}, the function
\begin{equation*}
\4{B}(x)=\left(\frac{C_{\ast}}{|x|^2}\right)^{\frac{1}{1-m}}
\end{equation*}
satisfies \eqref{elliptic-eqn} and 
\begin{equation*}
\lim_{|x|\to 0}|x|^{\frac{2}{1-m}}\4{B}(x)=C_{\ast}^{\frac{1}{1-m}}.
\end{equation*}
A natural question to 
ask is whether there exists a radially symmetric singular solution of \eqref{elliptic-eqn} that satisfies 
\begin{equation}\label{blow-up-rate-at-origin2}
\lim_{|x|\to 0}|x|^{\gamma}f(x)=\eta
\end{equation}
for some constant $0<\gamma\ne\alpha/\beta$, $\gamma\ne\frac{2}{1-m}$ and $\beta\ne 0$. We will answer this question in the negative. More precisely we will prove the following results in this paper.

\begin{thm}\label{existence-thm1}
Let $n\ge 3$, $0<m<\frac{n-2}{n}$, $0<\gamma\ne\frac{2}{1-m}$, $\eta>0$ and $\3>0$.  Suppose there exists  a radially symmetric solution $f$ of 
\begin{equation}\label{local-elliptic-eqn}
\Delta (f^m/m)+\alpha f+\beta x\cdot\nabla f=0,\quad f>0,\quad\mbox{ in }B_{\3}\setminus\{0\}
\end{equation}
which satisfies \eqref{blow-up-rate-at-origin2}  where $B_{\3}=\{x\in\mathbb{R}^n:|x|<\3\}$.  Then either 
\begin{enumerate}
\item[(i)] $\alpha=\beta=0$ holds \quad or 
\item[(ii)] $\alpha\in\mathbb{R}$ and $\beta\ne 0$  holds
with
\begin{equation}\label{alpha-beta-gamma-relation}
\gamma=\frac{\alpha}{\beta}\quad\mbox{ and }\quad\gamma>\frac{2}{1-m}.
\end{equation} 
\end{enumerate}
\end{thm}

\begin{thm}\label{existence-thm2}
Let $n\ge 3$, $0<m<\frac{n-2}{n}$, $\gamma=\frac{2}{1-m}$, $\alpha,\beta\in\mathbb{R}$, $\eta>0$ and $\3>0$. Suppose there exists  a radially symmetric solution $f$ of \eqref{local-elliptic-eqn} which satisfies \eqref{blow-up-rate-at-origin2}.  Then 
\begin{equation}\label{alpha-beta-relation5}
\left(\alpha-\frac{2\beta}{1-m}\right)\eta^{1-m}=\frac{2(n-2-nm)}{(1-m)^2}
\end{equation} 
holds.
\end{thm}

\begin{thm}\label{existence-thm3}
Let $n\ge 3$, $0<m<\frac{n-2}{n}$, $0<\gamma\ne\frac{2}{1-m}$ and $\eta>0$. Suppose there exists  a radially symmetric solution $f$ of  \eqref{elliptic-eqn} which satisfies \eqref{blow-up-rate-at-origin2}.  Then either 
\begin{enumerate}
\item[(i)] $\alpha=\beta=0$ holds \quad or 
\item[(ii)] $\alpha\in\mathbb{R}$ and $\beta\ne 0$  and \eqref{alpha-beta-gamma-relation} holds.
\end{enumerate}
\end{thm}

\begin{thm}\label{existence-thm4}
Let $n\ge 3$, $0<m<\frac{n-2}{n}$, $\gamma=\frac{2}{1-m}$ and $\eta>0$. Suppose there exists  a radially symmetric solution $f$ of \eqref{elliptic-eqn} which satisfies \eqref{blow-up-rate-at-origin2}.  Then 
\eqref{alpha-beta-relation5} holds.
\end{thm}

\begin{cor}\label{non-existence-cor2}
Let $n\ge 3$, $0<m<\frac{n-2}{n}$, $0<\gamma\ne\frac{2}{1-m}$ and $\eta>0$. Suppose $\alpha\in\mathbb{R}$ and $\beta\ne 0$ satisfy either \eqref{alpha-beta-forward-sss} or \eqref{alpha-beta-backward-sss} and $\gamma\ne\alpha/\beta$. Then the equation \eqref{elliptic-eqn} does not have any radially symmetric solution $f$ that satisfies \eqref{blow-up-rate-at-origin2}. 
\end{cor}

\begin{cor}\label{alpha-beta-neg-sign-cor}
Let $n\ge 3$, $0<m<\frac{n-2}{n}$, $0<\gamma\ne\frac{2}{1-m}$  and $\eta>0$. Let $\alpha\in\mathbb{R}$ and $\beta\ne 0$ satisfy \eqref{alpha-beta-forward-sss}. Suppose there exists a radially symmetric solution $f$ of \eqref{elliptic-eqn} which satisfies \eqref{blow-up-rate-at-origin2}. Then
\begin{equation}\label{alpha-beta-neg-sign}
\alpha<0\quad\mbox{ and }\quad\beta<0.
\end{equation}
\end{cor}

\begin{cor}\label{alpha-beta-positive-sign-cor}
Let $n\ge 3$, $0<m<\frac{n-2}{n}$, $0<\gamma\ne\frac{2}{1-m}$  and $\eta>0$. Let $\alpha\in\mathbb{R}$ and $\beta\ne 0$ satisfy \eqref{alpha-beta-backward-sss}. Suppose there exists a radially symmetric solution $f$ of \eqref{elliptic-eqn} which satisfies \eqref{blow-up-rate-at-origin2}. Then
\begin{equation}\label{alpha-beta-positive-sign}
\alpha>0\quad\mbox{ and }\quad\beta>0.
\end{equation}
\end{cor}

\section{Determination of admissible blow up rate}
\setcounter{equation}{0}
\setcounter{thm}{0}

In this section we will prove Theorem \ref{existence-thm1} and Theorem \ref{existence-thm2}. We first start with a lemma.

\begin{lem}\label{q-sequence-limit-lem}
Let $\gamma\in\mathbb{R}$, $\eta>0$ and  $\3>0$. For any $\alpha, \beta\in\mathbb{R}$, let $f$ be a radially symmetric solution of \eqref{local-elliptic-eqn} which satisfies \eqref{blow-up-rate-at-origin2}. Let
\begin{equation}\label{q-defn}
q(r)=\frac{rf_r(r)}{f(r)}\quad\forall 0<r<\3.
\end{equation}
Then there exists a decreasing sequence $\{r_i\}_{i=1}^{\infty}\subset(0,\3)$, $r_i\to 0$ as $i\to\infty$, that satisfies
\begin{equation}\label{q-squence-limit}
\lim_{i\to\infty}q(r_i)=-\gamma.
\end{equation}
\end{lem}
\begin{proof}
Let $w(r,t)=r^{\gamma}f(r)$. By direct computation,
\begin{equation}\label{w-q-eqn}
w_r(r)=\frac{w(r)}{r}(q(r)+\gamma)\quad\Rightarrow\quad q(r)+\gamma=\frac{rw_r(r)}{w(r)}\quad\forall 0<r<\3.
\end{equation}
Then by \eqref{blow-up-rate-at-origin2}, \eqref{w-q-eqn} and the same argument as the proof of Theorem 1.4 of \cite{H4} there exists a decreasing sequence $\{r_i\}_{i=1}^{\infty}\subset(0,\3)$, $r_i\to 0$ as $i\to\infty$, that satisfies
\begin{equation*}
\lim_{i\to\infty}|q(r_i)+\gamma|=0
\end{equation*}
and the lemma follows.
\end{proof}

\noindent{\ni{\it Proof of Theorem \ref{existence-thm1}:}} 

\noindent Let $q$ be given by \eqref{q-defn}. By direct computation $q(r)$ satisfies
\begin{equation}\label{q-eqn}
q_r+\left(\frac{n-2}{r}+\beta rf(r)^{1-m}\right)q(r)=-\alpha rf(r)^{1-m}-\frac{m}{r}q(r)^2\quad\forall 0<r<\3.
\end{equation}
Let
\begin{equation}\label{F-defn}
F(r)=\exp\left(-\beta\int_r^{\3/2}\rho f(\rho)^{1-m}\,d\rho\right)\quad\forall 0<r<\3/2.
\end{equation}
Multiplying \eqref{q-eqn} by $r^{n-2}F(r)$ and integrating over $(r,\3/2)$, $0<r<\3/2$, we get
\begin{equation}\label{q-integral-relation}
r^{n-2}F(r)q(r)=(\3/2)^{n-2}q(\3/2)+\int_r^{\3/2}\left(\alpha\rho^{n-1} f(\rho)^{1-m}+m\rho^{n-3}q(\rho)^2\right)F(\rho)\,d\rho
\quad\forall 0<r<\3/2.
\end{equation}
By Lemma \ref{q-sequence-limit-lem}  there exists a decreasing sequence $\{r_i\}_{i=1}^{\infty}\subset(0,\3/2)$, $r_i\to 0$ as $i\to\infty$, that satisfies \eqref{q-squence-limit}. By \eqref{blow-up-rate-at-origin2} 
there exists a constant $0<r_0<\3$ such that
\begin{equation}\label{f-decay-origin-estimate}
\frac{\eta}{2}r^{-\gamma}\le f(r)\le 2\eta r^{-\gamma}\quad\forall 0<r\le r_0.
\end{equation}
We now divide the proof  into two cases.

\noindent $\underline{\text{\bf Case 1}}$: $\beta\ge 0$.

\noindent By \eqref{F-defn},
\begin{equation}\label{F-upper-estimate}
F(r)=F(r_0)\exp\left(-\beta\int_r^{r_0}\rho f(\rho)^{1-m}\,d\rho\right)
\le F(r_0)\quad\forall 0<r<r_0.
\end{equation}
Hence by \eqref{F-upper-estimate},
\begin{equation}\label{F-limit3}
r^{n-2}F(r)\to 0\quad\mbox{ as }r\to 0\quad\mbox{if }\beta\ge 0.
\end{equation}

\noindent $\underline{\text{\bf Case 2}}$: $\beta<0$.

\noindent \noindent By \eqref{F-defn} and \eqref{f-decay-origin-estimate},
\begin{align}\label{F-lower-estimate}
F(r)=&F(r_0)\exp\left(-\beta\int_r^{r_0}\rho f(\rho)^{1-m}\,d\rho\right)\notag\\
\ge&F(r_0)\exp\left(\frac{|\beta|\eta^{1-m}}{2^{1-m}}\int_r^{r_0}\rho^{1-\gamma (1-m)}\,d\rho\right)\notag\\
=&F(r_0)\exp\left(\frac{|\beta|\eta^{1-m}\left(r^{2-\gamma (1-m)}-r_0^{2-\gamma (1-m)}\right)}{2^{1-m}(\gamma(1-m)-2 )}\right)\quad\forall 0<r<r_0\quad\mbox{ if }\gamma>\frac{2}{1-m}.
\end{align}
and
\begin{align}\label{F-upper-estimate2}
F(r)=&F(r_0)\exp\left(-\beta\int_r^{r_0}\rho f(\rho)^{1-m}\,d\rho\right)\notag\\
\le&F(r_0)\exp\left((2\eta)^{1-m}|\beta|\int_r^{r_0}\rho^{1-\gamma (1-m)}\,d\rho\right)\notag\\
=&F(r_0)\exp\left(\frac{(2\eta)^{1-m}|\beta|\left(r_0^{2-\gamma (1-m)}-r^{2-\gamma (1-m)}\right)}{2-\gamma(1-m)}\right)\notag\\
\le&F(r_0)\exp\left(\frac{(2\eta)^{1-m}|\beta| r_0^{2-\gamma (1-m)}}{2-\gamma(1-m)}\right)\quad\forall 0<r<r_0\quad\mbox{ if }0<\gamma<\frac{2}{1-m}.
\end{align}
Hence by \eqref{F-lower-estimate} and \eqref{F-upper-estimate2},
\begin{equation}\label{F-limit6}
\left\{\begin{aligned}
&r^{n-2}F(r)\to\infty\quad\mbox{ as }r\to 0\quad\mbox{ if }\beta<0\mbox{ and }\gamma>\frac{2}{1-m}\\
&r^{n-2}F(r)\to 0\quad\mbox{ as }r\to 0\quad\mbox{ if }\beta<0\mbox{ and }0<\gamma<\frac{2}{1-m}.
\end{aligned}\right.
\end{equation}
By \eqref{q-squence-limit}, \eqref{q-integral-relation}, \eqref{F-limit3} and \eqref{F-limit6},
\begin{align}\label{num-demin-infty-zero2}
&(\3/2)^{n-2}q(\3/2)+\int_{r_i}^{\3/2}\left(\alpha\rho^{n-1} f(\rho)^{1-m}+m\rho^{n-3}q(\rho)^2\right)F(\rho)\,d\rho\notag\\
=&r_i^{n-2}F(r_i)q(r_i)\to\left\{\begin{aligned}
&0\qquad\mbox{ as }i\to\infty\quad\mbox{ if }\beta\ge 0\\
&0\qquad\mbox{ as }i\to\infty\quad\mbox{ if }\beta<0\mbox{ and }0<\gamma<\frac{2}{1-m}\\
&-\infty\quad\mbox{ as }i\to\infty\quad\mbox{ if }\beta<0\mbox{ and }\gamma>\frac{2}{1-m}.
\end{aligned}\right.
\end{align}
Hence by  \eqref{q-squence-limit}, \eqref{q-integral-relation}, \eqref{num-demin-infty-zero2} and the l'Hospital rule,
\begin{align}\label{limit-q1}
-\gamma=\lim_{i\to\infty}q(r_i)=&\lim_{i\to\infty}\frac{(\3/2)^{n-2}q(\3/2)+\int_{r_i}^{\3/2}\left(\alpha\rho^{n-1} f(\rho)^{1-m}+m\rho^{n-3}q(\rho)^2\right)F(\rho)\,d\rho}{r_i^{n-2}F(r_i)}\notag\\
=&-\lim_{i\to\infty}\frac{\left(\alpha r_i^{n-1}f(r_i)^{1-m}+mr_i^{n-3}q(r_i)^2\right)F(r_i)}{(n-2)r_i^{n-3}F(r_i)
+r_i^{n-2}F(r_i)\cdot\beta r_if(r_i)^{1-m}}\notag\\
=&-\lim_{i\to\infty}\frac{\alpha r_i^2f(r_i)^{1-m}+m\gamma^2}{n-2+\beta r_i^2f(r_i)^{1-m}}.
\end{align}
Now by \eqref{blow-up-rate-at-origin2},
\begin{equation}\label{f-limit7}
\left\{\begin{aligned}
&r_i^2f(r_i)^{1-m}\to\infty\quad\mbox{ as }i\to\infty\quad\mbox{ if }\gamma>\frac{2}{1-m}\\
&r_i^2f(r_i)^{1-m}\to 0\quad\mbox{ as }i\to\infty\quad\mbox{ if }0<\gamma<\frac{2}{1-m}.
\end{aligned}\right.
\end{equation}
Hence if $0<\gamma<\frac{2}{1-m}$, then by \eqref{limit-q1} and \eqref{f-limit7} we get
\begin{equation*}
\gamma=\frac{m\gamma^2}{n-2}\quad\Rightarrow\quad\gamma=\frac{n-2}{m}>\frac{2}{1-m}
\end{equation*}
and contradiction arises. Thus $\gamma>\frac{2}{1-m}$. 

Suppose $\beta=0$ and $\alpha\ne 0$ holds. Then \eqref{limit-q1} and \eqref{f-limit7}, we get $\gamma=\pm\infty$ and contradiction arises. Hence either (i) holds or  
\begin{equation}\label{alpha-beta-relation7}
\beta\ne 0\quad\mbox{ and }\quad\alpha\in\mathbb{R}
\end{equation}
holds. Suppose \eqref{alpha-beta-relation7} holds. Then by \eqref{limit-q1} and \eqref{f-limit7}, 
we get \eqref{alpha-beta-gamma-relation} and the theorem follows.

{\hfill$\square$\vspace{6pt}}

\noindent{\ni{\it Proof of Theorem \ref{existence-thm2}:}} 

\noindent Let $q(r)$ be given by \eqref{q-defn}, $\gamma=\frac{2}{1-m}$ and $F$ be given by \eqref{F-defn}. Let $r_0\in (0,\3)$ and $\{r_i\}_{i=1}^{\infty}$ be as in the proof of Theorem \ref{existence-thm1}. Then  \eqref{q-integral-relation}
and \eqref{f-decay-origin-estimate} holds.

We now divide the proof  into three cases.

\noindent $\underline{\text{\bf Case 1}}$: $\beta\ge 0$.

\noindent By \eqref{F-defn}, we get \eqref{F-upper-estimate}. Hence by \eqref{F-upper-estimate}, \eqref{F-limit3} holds.

\noindent $\underline{\text{\bf Case 2}}$: $\beta<0$ and $n-2-|\beta|\eta^{1-m}<0$.

\noindent We choose $a_1\in (0,1)$ such that $n-2-(a_1\eta)^{1-m}|\beta|<0$.
By \eqref{blow-up-rate-at-origin2} there exists a constant $0<b_0<r_0$ such that
\begin{equation}\label{f-decay-origin-estimate2}
a_1\eta r^{-\gamma}\le f(r)\le 2\eta r^{-\gamma}\quad\forall 0<r\le b_0.
\end{equation}
By \eqref{F-defn} and \eqref{f-decay-origin-estimate2},
\begin{align}\label{F-lower-estimate3}
F(r)=&F(b_0)\exp\left(-\beta\int_r^{b_0}\rho f(\rho)^{1-m}\,d\rho\right)\notag\\
\ge&F(b_0)\exp\left((a_1\eta)^{1-m}|\beta|\int_r^{b_0}\rho^{1-\gamma (1-m)}\,d\rho\right)\notag\\
\ge&F(b_0)\exp\left((a_1\eta)^{1-m}|\beta|\log(r_0'/r)\right)\notag\\
=&F(b_0)\left(\frac{b_0}{r}\right)^{(a_1\eta)^{1-m}|\beta|}\quad\forall 0<r<b_0.
\end{align}
Hence by \eqref{F-lower-estimate3},
\begin{equation}\label{F-limit8}
r^{n-2}F(r)\ge F(b_0)b_0^{(a_1\eta)^{1-m}|\beta|}r^{n-2-(a_1\eta)^{1-m}|\beta|}\to\infty\quad\mbox{ as }r\to 0.
\end{equation}
By \eqref{q-squence-limit}, \eqref{q-integral-relation}, \eqref{F-limit3} and \eqref{F-limit8}, for case 1 and case 2 we have
\begin{align}\label{num-demin-infty-zero3}
&(\3/2)^{n-2}q(\3/2)+\int_{r_i}^{\3/2}\left(\alpha\rho^{n-1} f(\rho)^{1-m}+m\rho^{n-3}q(\rho)^2\right)F(\rho)\,d\rho\notag\\
=&r_i^{n-2}F(r_i)q(r_i)\to -\infty\mbox{ or }0\quad\mbox{ as }i\to\infty.
\end{align}
Hence by  \eqref{q-squence-limit}, \eqref{q-integral-relation},  \eqref{num-demin-infty-zero3} and the l'Hospital's rule, for case 1 and case 2 we get \eqref{limit-q1}. Since $\gamma=\frac{2}{1-m}$, by \eqref{blow-up-rate-at-origin2} and \eqref{limit-q1},
\begin{equation*}
\frac{2}{1-m}\cdot(n-2+\beta\eta^{1-m})=\alpha\eta^{1-m}+\frac{4m}{(1-m)^2}\quad
\Rightarrow\quad\left(\alpha-\frac{2\beta}{1-m}\right)\eta^{1-m}=\frac{2(n-2-nm)}{(1-m)^2}
\end{equation*}
and \eqref{alpha-beta-relation5} follows.

\noindent $\underline{\text{\bf Case 3}}$: $\beta<0$ and $n-2-|\beta|\eta^{1-m}\ge 0$.

\noindent Since $n-1-|\beta|\eta^{1-m}\ge 1$, there exists a constant $a_2>1$ such that $n-1-(a_2\eta)^{1-m}|\beta|>0$.
By \eqref{blow-up-rate-at-origin2} there exists a constant $0<b_1<r_0$ such that
\begin{equation}\label{f-decay-origin-estimate4}
a_1\eta r^{-\gamma}\le f(r)\le a_2\eta r^{-\gamma}\quad\forall 0<r\le b_1.
\end{equation}
By \eqref{F-defn} and \eqref{f-decay-origin-estimate4},
\begin{align}\label{F-lower-estimate13}
F(r)\le&F(b_1)\exp\left((a_2\eta)^{1-m}|\beta|\int_r^{b_1}\rho^{1-\gamma (1-m)}\,d\rho\right)\notag\\
\le&F(b_1)\exp\left((a_2\eta)^{1-m}|\beta|\log(b_1/r)\right)\notag\\
=&F(b_1)\left(\frac{b_1}{r}\right)^{(a_2\eta)^{1-m}|\beta|}\quad\forall 0<r<b_1.
\end{align}
Then  by \eqref{F-lower-estimate13}  we have
\begin{equation}\label{F-limit11}
r^{n-1}F(r)\le F(b_1)b_1^{(a_2\eta)^{1-m}|\beta|}r^{n-1-(a_2\eta)^{1-m}|\beta|}\to 0\quad\mbox{ as }r\to 0.
\end{equation}
By \eqref{q-squence-limit}, \eqref{q-integral-relation}, \eqref{F-limit11},  we have
\begin{align}\label{num-demin-infty-zero6}
&(\3/2)^{n-2}q(\3/2)r_i+r_i\int_{r_i}^{\3/2}\left(\alpha\rho^{n-1} f(\rho)^{1-m}+m\rho^{n-3}q(\rho)^2\right)F(\rho)\,d\rho\notag\\
=&r_i^{n-1}F(r_i)q(r_i)\to 0\quad\mbox{ as }i\to\infty.
\end{align}
Since $\gamma=\frac{2}{1-m}$, by  \eqref{blow-up-rate-at-origin2}, \eqref{q-squence-limit}, \eqref{q-integral-relation},  \eqref{num-demin-infty-zero6} and the l'Hospital rule, 
\begin{align}\label{limit-q5}
-\gamma=\lim_{i\to\infty}q(r_i)=&\lim_{i\to\infty}\frac{(\3/2)^{n-2}q(\3/2)r_i+r_i\int_{r_i}^{\3/2}\left(\alpha\rho^{n-1} f(\rho)^{1-m}+m\rho^{n-3}q(\rho)^2\right)F(\rho)\,d\rho}{r_i^{n-1}F(r_i)}\notag\\
=&\lim_{i\to\infty}\frac{(\3/2)^{n-2}q(\3/2)+\int_{r_i}^{\3/2}\left(\alpha\rho^{n-1} f(\rho)^{1-m}+m\rho^{n-3}q(\rho)^2\right)F(\rho)\,d\rho}{(n-1)r_i^{n-2}F(r_i)
+r_i^{n-1}F(r_i)\cdot\beta r_if(r_i)^{1-m}}\notag\\
&\quad -\lim_{i\to\infty}\frac{\left(\alpha r_i^{n-1}f(r_i)^{1-m}+mr_i^{n-3}q(r_i)^2\right)r_iF(r_i)}{(n-1)r_i^{n-2}F(r_i)+r_i^{n-1}F(r_i)\cdot\beta r_if(r_i)^{1-m}}\notag\\
=&\lim_{i\to\infty}\frac{(\3/2)^{n-2}q(\3/2)+\int_{r_i}^{\3/2}\left(\alpha\rho^{n-1} f(\rho)^{1-m}+m\rho^{n-3}q(\rho)^2\right)F(\rho)\,d\rho}{r_i^{n-2}F(r_i)\left(n-1
+\beta r_i^2f(r_i)^{1-m}\right)}\notag\\
&\quad -\lim_{i\to\infty}\frac{\alpha r_i^2f(r_i)^{1-m}+m\gamma^2}{n-1+\beta r_i^2f(r_i)^{1-m}}\notag\\
=&-\frac{\gamma+\alpha\eta^{1-m}+m\gamma^2}{n-1+\beta\eta^{1-m}}
\end{align}
and \eqref{alpha-beta-relation5} follows.

{\hfill$\square$\vspace{6pt}}

Note that Theorem \ref{existence-thm3}, Corollary \ref{non-existence-cor2} and Theorem \ref{existence-thm4} follows directly from Theorem \ref{existence-thm1} and Theorem \ref{existence-thm2}. 

\noindent{\ni{\it Proof of Corollary \ref{alpha-beta-neg-sign-cor}:}} 

\noindent By Theorem \ref{existence-thm1}, \eqref{alpha-beta-gamma-relation} holds. Then by \eqref{alpha-beta-forward-sss} and \eqref{alpha-beta-gamma-relation},
\begin{equation*}
\frac{2\beta-1}{\beta(1-m)}>\frac{2}{1-m}\quad\Rightarrow\quad\beta<0\quad\Rightarrow\quad\alpha<0 
\end{equation*}
and \eqref{alpha-beta-neg-sign} follows.

{\hfill$\square$\vspace{6pt}}

\noindent{\ni{\it Proof of Corollary \ref{alpha-beta-positive-sign-cor}:}} 

\noindent By Theorem \ref{existence-thm1}, \eqref{alpha-beta-gamma-relation} holds. Then by \eqref{alpha-beta-backward-sss} and \eqref{alpha-beta-gamma-relation},
\begin{equation*}
\frac{2\beta+1}{\beta(1-m)}>\frac{2}{1-m}\quad\Rightarrow\quad\beta>0\quad\Rightarrow\quad\alpha>0 
\end{equation*}
and \eqref{alpha-beta-positive-sign} follows.

{\hfill$\square$\vspace{6pt}}

\end{document}